\documentclass[a4paper,11pt, reqno]{amsart}

\usepackage{tikz}
\usepackage{subfigure}
\usepackage[final]{microtype}
\usepackage{lmodern}
\usepackage[normalem]{ulem}

\usepackage{amsmath, amssymb, amsfonts, amsthm}
\usepackage[latin1]{inputenc}
\usepackage[T1]{fontenc}

\usepackage[draft=false,pdftex]{hyperref}
\hypersetup{
    bookmarks=true,         
    unicode=false,          
    pdftoolbar=true,        
    pdfmenubar=true,        
    pdffitwindow=false,     
    pdfstartview={FitH},    
    pdftitle={The number of degree 5 vertices in a matchstick graph},    
    pdfauthor={Jérémy Lavollée and Konrad J. Swanepoel},     
    pdfnewwindow=false,      
    colorlinks=true,       
    linkcolor=black,          
    citecolor=black,        
    filecolor=black,      
    urlcolor=black           
}

\theoremstyle{plain}
\newtheorem{theorem}{Theorem}
\newtheorem{lemma}[theorem]{Lemma}
\newtheorem{corollary}[theorem]{Corollary}

\theoremstyle{remark}
\newtheorem{remark}[theorem]{Remark}

\definecolor{konrad}{rgb}{0,0,1}
\definecolor{jeremy}{rgb}{1,0,0}

\renewcommand{\geq}{\geqslant}
\renewcommand{\leq}{\leqslant}

\begin{document}

\title{The number of small-degree vertices in matchstick graphs}
\author{J\'er\'emy Lavoll\'ee}
\address[J. Lavoll\'ee]{Department of Mathematics, Universitat Polit\`{e}cnica de Catalunya, Spain.}

\author{Konrad J. Swanepoel}
\address[K. J. Swanepoel]{Department of Mathematics, The London School of Economics and Political Science, Houghton Street, London WC2A 2AE, U.K.}

\date{}

\begin{abstract}
A matchstick graph is a crossing-free unit-distance graph in the plane.
Harborth (1981) proposed the problem of determining whether there exists a matchstick graph in which every vertex has degree exactly $5$.
In 1982, Blokhuis gave a proof of non-existence.
A shorter proof was found by Kurz and Pinchasi (2011) using a discharging method.
We combine their method with the isoperimetric inequality to show that there are $\Omega(\sqrt{n})$ vertices in a matchstick graph on $n$ vertices that are of degree at most $4$, which is asymptotically tight.
\end{abstract}

\maketitle

\section{Introduction}
Matchstick graphs are graphs which can be drawn in the plane with edges represented by unit length straight-line segments that intersect only at their endpoints.
Harborth \cite{oberwolfach, lighter} introduced these in 1981 and posed the problem of finding the least number of vertices for a $k$-regular matchstick graph.
In 1982, Blokhuis \cite{blokhuis} proved that no 5-regular matchstick graph exists.
Kurz and Pinchasi \cite{kurz} gave a short proof of the same result by a discharging method.

In this paper we find a stronger result by considering the vertices of degree at most $4$ in a matchstick graph with no isolated vertices.
\begin{theorem}\label{thm:small-deg}
    For any matchstick graph on $n$ vertices with no isolated vertices and $n_i$ vertices of degree $i$, we have $4n_1+3n_2+2n_3+n_4 > \sqrt{\pi \sqrt{3}\cdot n}$.
\end{theorem}
The proof combines Kurz and Pinchasi's discharging method with the isoperimetric inequality.
The main result of this paper then follows as a corollary.
\begin{corollary}
    \label{co:n_5}
    In a matchstick graph on $n$ vertices, the number of vertices of degree at most $4$ is greater than $\frac{1}{4} \sqrt{\pi\sqrt{3}\cdot n}$.
\end{corollary}
\begin{remark}\label{remark3}
    Theorem~\ref{thm:small-deg} and Corollary~\ref{co:n_5} are sharp up to the constant.
    As shown in Figure~\ref{fig:tiling}, it is possible to construct matchstick graphs with $n=3k^2+3k$ vertices, of which all are of degree $5$ except for six of degree $3$ and $6(k-1)$ of degree $4$.
    The construction consists of a hexagonal piece of the triangular lattice with the central vertex removed, as well as removing every second edge of the boundary of each hexagon except for the smallest and the largest hexagon.
    In this case, $\frac{2n_3+n_4}{\sqrt{n}}=2\sqrt{3}\sqrt{1+\frac1k}$
    is quite far from $\sqrt{\pi\sqrt{3}}$ in Theorem~\ref{thm:small-deg}, partly because for large $k$ the boundary of this example is a regular hexagon instead of circular, and partly because in the inequality $b^2 > \pi\sqrt{3} f_3$, we only consider triangles to bound the area from below.
    In this example there are many quadrilaterals, and it is not clear how to get a non-trivial lower bound for the area of many quadrilaterals in a matchstick graph.
\end{remark}

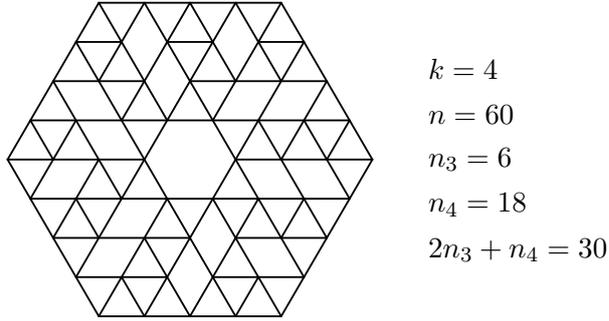
\begin{figure}
\centering
\begin{tikzpicture}[line cap=round,line join=round,x=1.2cm,y=1.2cm,rotate=90]
\clip(-2,-5) rectangle (2,5);
\draw [line width=0.7pt] (0,0.5)-- (0,2);
\draw [line width=0.7pt] (0,2)-- (1.7320508075688772,1);
\draw [line width=0.7pt] (1.7320508075688772,1)-- (1.7320508075688772,-1);
\draw [line width=0.7pt] (0,1.5)-- (0.4330127018922193,1.75);
\draw [line width=0.7pt] (0,1)-- (0.8660254037844386,1.5);
\draw [line width=0.7pt] (0,0.5)-- (1.299038105676658,1.25);
\draw [line width=0.7pt] (0,0.5)-- (0.4330127018922193,0.25);
\draw [line width=0.7pt] (0.4330127018922193,0.25)-- (0.4330127018922193,-0.25);
\draw [line width=0.7pt] (0.4330127018922193,-0.25)-- (1.7320508075688772,0.5);
\draw [line width=0.7pt] (0.4330127018922193,0.25)-- (1.7320508075688772,1);
\draw [line width=0.7pt] (0.8660254037844386,-0.5)-- (1.7320508075688772,0);
\draw [line width=0.7pt] (1.299038105676658,-0.75)-- (1.7320508075688772,-0.5);
\draw [line width=0.7pt] (0.4330127018922193,0.25)-- (1.7320508075688772,-0.5);
\draw [line width=0.7pt] (0.4330127018922193,0.25)-- (0.4330127018922193,1.75);
\draw [line width=0.7pt] (0,1.5)-- (0.4330127018922193,1.25);
\draw [line width=0.7pt] (0.4330127018922193,0.75)-- (1.7320508075688772,0);
\draw [line width=0.7pt] (0.8660254037844386,1.5)-- (0.8660254037844386,0.5);
\draw [line width=0.7pt] (0.8660254037844386,0)-- (0.8660254037844386,-0.5);
\draw [line width=0.7pt] (0.8660254037844386,1)-- (1.7320508075688772,0.5);
\draw [line width=0.7pt] (1.299038105676658,1.25)-- (1.299038105676658,0.75);
\draw [line width=0.7pt] (1.299038105676658,0.25)-- (1.299038105676658,-0.25);
\draw [line width=0.7pt] (-0.43301270189221935,-0.25)-- (-1.7320508075688774,-1);
\draw [line width=0.7pt] (-1.7320508075688774,-1)-- (-1.732050807568877,1);
\draw [line width=0.7pt] (-1.732050807568877,1)-- (0,2);
\draw [line width=0.7pt] (-1.299038105676658,-0.75)-- (-1.7320508075688774,-0.5);
\draw [line width=0.7pt] (-0.8660254037844387,-0.5)-- (-1.7320508075688772,0);
\draw [line width=0.7pt] (-0.43301270189221935,-0.25)-- (-1.7320508075688772,0.5);
\draw [line width=0.7pt] (-0.43301270189221935,-0.25)-- (-0.43301270189221924,0.25);
\draw [line width=0.7pt] (-0.43301270189221924,0.25)-- (0,0.5);
\draw [line width=0.7pt] (0,0.5)-- (-1.2990381056766576,1.25);
\draw [line width=0.7pt] (-0.43301270189221924,0.25)-- (-1.732050807568877,1);
\draw [line width=0.7pt] (0,1)-- (-0.8660254037844383,1.5);
\draw [line width=0.7pt] (0,1.5)-- (-0.4330127018922189,1.75);
\draw [line width=0.7pt] (-0.43301270189221924,0.25)-- (-0.4330127018922189,1.75);
\draw [line width=0.7pt] (-0.43301270189221924,0.25)-- (-1.7320508075688774,-0.5);
\draw [line width=0.7pt] (-1.299038105676658,-0.75)-- (-1.299038105676658,-0.25);
\draw [line width=0.7pt] (-0.8660254037844386,0)-- (-0.8660254037844383,1.5);
\draw [line width=0.7pt] (-1.7320508075688772,0)-- (-0.8660254037844385,0.5);
\draw [line width=0.7pt] (-0.43301270189221913,0.75)-- (0,1);
\draw [line width=0.7pt] (-1.2990381056766578,0.25)-- (-1.2990381056766576,1.25);
\draw [line width=0.7pt] (-1.7320508075688772,0.5)-- (-1.2990381056766578,0.75);
\draw [line width=0.7pt] (-0.8660254037844383,1)-- (-0.4330127018922191,1.25);
\draw [line width=0.7pt] (0.43301270189221935,-0.25)-- (1.7320508075688774,-1);
\draw [line width=0.7pt] (1.7320508075688774,-1)-- (0,-2);
\draw [line width=0.7pt] (0,-2)-- (-1.732050807568877,-1);
\draw [line width=0.7pt] (1.299038105676658,-0.75)-- (1.2990381056766582,-1.25);
\draw [line width=0.7pt] (0.8660254037844387,-0.5)-- (0.8660254037844388,-1.5);
\draw [line width=0.7pt] (0.43301270189221935,-0.25)-- (0.43301270189221974,-1.75);
\draw [line width=0.7pt] (0.43301270189221935,-0.25)-- (0,-0.5);
\draw [line width=0.7pt] (0,-0.5)-- (-0.43301270189221924,-0.25);
\draw [line width=0.7pt] (-0.43301270189221924,-0.25)-- (-0.4330127018922189,-1.75);
\draw [line width=0.7pt] (0,-0.5)-- (0,-2);
\draw [line width=0.7pt] (-0.8660254037844385,-0.5)-- (-0.8660254037844383,-1.5);
\draw [line width=0.7pt] (-1.2990381056766576,-0.75)-- (-1.2990381056766576,-1.25);
\draw [line width=0.7pt] (0,-0.5)-- (-1.2990381056766576,-1.25);
\draw [line width=0.7pt] (0,-0.5)-- (1.2990381056766582,-1.25);
\draw [line width=0.7pt] (1.299038105676658,-0.75)-- (0.8660254037844389,-1);
\draw [line width=0.7pt] (0.4330127018922194,-0.75)-- (-0.8660254037844383,-1.5);
\draw [line width=0.7pt] (0.8660254037844388,-1.5)-- (0,-1);
\draw [line width=0.7pt] (-0.43301270189221913,-0.75)-- (-0.8660254037844385,-0.5);
\draw [line width=0.7pt] (0.43301270189221946,-1.25)-- (-0.4330127018922189,-1.75);
\draw [line width=0.7pt] (0.43301270189221974,-1.75)-- (0,-1.5);
\draw [line width=0.7pt] (-0.43301270189221897,-1.25)-- (-0.8660254037844383,-1);
\begin{scope}
\draw (1,-2.5) node[right]{$k=4$};
\draw (0.5,-2.5) node[right]{$n=60$};
\draw (0,-2.5) node[right]{$n_3=6$};
\draw (-0.5,-2.5) node[right]{$n_4=18$};
\draw (-1,-2.5) node[right]{$2n_3+n_4=30$};
\end{scope}
\end{tikzpicture}
        \caption{Matchstick graphs with $n=3k^2+3k$ vertices, $n_3=6$, $n_4=6(k-1)$ and $n_5=3k^2-3k$.}
        \label{fig:tiling}
\end{figure}

\section{Plane graphs and matchstick graphs}\label{planegraphs}
A \emph{plane graph} $G=(V,E)$ is a drawing of a graph in the plane where all vertices $v \in V$ are distinct points and all edges $uv \in E$ are Jordan arcs joining $u$ and $v$, such that edges intersect only at their endpoints.
A \emph{matchstick graph} is a plane graph with edges drawn as unit length straight-line segments.

We denote the number of vertices by $n$, the number of vertices of degree $i$ by $n_i$, and the number of edges by $e$.
Then $n=\sum_i n_i$ and $2e=\sum_i in_i$.
We next find similar formulas for the total number of faces and the number of faces with $k$ sides.
We need to be careful in how we define the number of sides of a face.
Note that although $G$ is planar, we do not assume that $G$ is $2$-connected, nor even that it is connected.
This is to avoid a complicated induction in the proof of Theorem~\ref{thm:small-deg}.
As usual, we define a \emph{face} of $G$ to be a connected component of the complement of the drawing of $G$ in the plane.
Thus there is a unique unbounded face, and $0$ or more bounded faces.
We say that a face has $k$ sides or is a $k$-gon if the face has $k$ edges on its boundary, where we count an edge twice if both sides of the edge are in the face.
For example, the matchstick graph in Figure~\ref{fig:2edge} has an unbounded face with $4$ sides and a bounded face with $6$ sides, since the edge $e$ counts twice.
\begin{figure}[h]
    \centering
    \begin{tikzpicture}[line cap=round,line join=round,x=1.8cm,y=1.8cm]
    \clip(-1.2006844359112205,-2.4510316746875613) rectangle (0.8027537263570521,-1.3513587150125688);
    \draw [line width=1.2pt] (-1.0684022540606486,-1.8381065002148687)-- (-0.15608621114048593,-1.4286197020998428);
    \draw [line width=1.2pt] (-0.15608621114048593,-1.4286197020998428)-- (0.6715639904230459,-1.9898639838986376);
    \draw [line width=1.2pt] (0.6715639904230459,-1.9898639838986376)-- (-0.2423647781732745,-2.395738602385694);
    \draw [line width=1.2pt] (-1.0684022540606486,-1.8381065002148687)-- (-0.07358793703668776,-1.7363987097357207);
    \draw [line width=1.2pt] (-1.0684022540606486,-1.8381065002148687)-- (-0.2423647781732745,-2.395738602385694);
    \draw [fill=black] (-1.0684022540606486,-1.8381065002148687) circle (1.5pt);
    \draw [fill=black] (-0.15608621114048593,-1.4286197020998428) circle (1.5pt);
    \draw [fill=black] (0.6715639904230459,-1.9898639838986376) circle (1.5pt);
    \draw [fill=black] (-0.2423647781732745,-2.395738602385694) circle (1.5pt);
    \draw [fill=black] (-0.07358793703668776,-1.7363987097357207) circle (1.5pt);
    \draw[color=black] (-0.5080369984083286,-1.910430638754684) node {$e$};
    \end{tikzpicture}
    \caption{The bounded face has $6$ sides with $e$ counting twice.}
    \label{fig:2edge}
\end{figure}
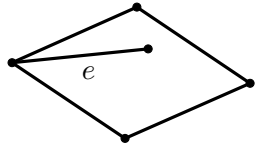
We let $b$ denote the number of sides of the unbounded face, $f_k$ the number of bounded $k$-gons, and $f$ the total number of bounded faces.
Then we have $f=\sum_k f_k$, and by counting edge-face pairs, we obtain $2e=b+\sum_k kf_k$.

Let us consider the possible $k$-gons for some small $k$.
This will be required in the proof of Theorem \ref{thm:small-deg}.
The only way for a $0$-gon to exist is if the graph has no edges.
Since the graph is simple, a $1$-gon cannot exist and the only way for a $2$-gon to exist is if the graph has a single edge.
The only possible $3$-gon in a matchstick graph is an equilateral triangle.
A $4$-gon is an equilateral rhombus unless the graph has only two edges with a common vertex.
Similarly, the only possible bounded $5$-gon is an equilateral pentagon, since the only other combinatorial possibility is a triangle together with a vertex of degree~$1$ joined to one of the vertices of the triangle.
This can only happen if there are exactly $4$ edges, and the fourth edge is outside the triangle.

An \emph{isolated vertex} is one with degree~$0$.
Note that isolated vertices in a matchstick graph  give no structural information about the edges and can simply be removed from the graph.

\section{Proofs}
The isoperimetric inequality says that among all simple closed curves of a fixed length in the plane, the circle is the unique curve that encloses the largest area \cite{blasjo}.
In the proof of Theorem~\ref{thm:small-deg} we will need the following consequence.
The $2$-connected case of this lemma is used in \cite{lavollee}.

\begin{lemma}\label{isoperimetric-general}
Let $G$ be a matchstick graph such that the unbounded face has $b$ edges (with double counting as before) and $f_3$ bounded triangular faces.
Then $b^2 > \pi\sqrt{3}f_3$.
\end{lemma}
\begin{proof}
First consider the case where $G$ is 2-connected.
The polygon bounding the unbounded face has $b$ edges and encloses $f_3$ equilateral triangles of area $\sqrt{3}/4$ each.
Thus, the polygon has area $A\geq \frac{\sqrt{3}}{4}f_3$ and perimeter $b$.
By the isoperimetric inequality, 
$b^2 > 4\pi A\geq \pi\sqrt{3}f_3$.

Now assume that $G$ is not 2-connected.
If $b=2$ then $f_3=0$ and the lemma is trivially true.
We now use induction on $b$.
Thus we assume that $b\geq 3$ and that the lemma is true for smaller $b$.

Since $G$ is not $2$-connected there is either a cut-point $v$, since $b \geq 3$, or $G$ is not connected.
In either case we can split $G$ into subgraphs $G'$ and $G''$ with at most the vertex $v$ in common and with at least two vertices each.
Let $b'$ and $b''$ denote the number of edges on their boundaries and $f_3'$ and $f_3''$ the number of triangular faces respectively.
It is clear that $b=b'+b''$ and $f_3=f_3'+f_3''$, so by induction
\begin{equation*}
    b^2 \geq b'^2+b''^2
    > \pi\sqrt{3}f'_3 +  \pi\sqrt{3}f''_3
    = \pi\sqrt{3}f_3. \qedhere
\end{equation*}

\end{proof}

\smallskip
The proof of Theorem~\ref{thm:small-deg} below uses the discharging method from \cite{kurz}.
The discharging method is used in graph theory and combinatorial geometry to prove results about planar graphs, geometric graphs and line arrangements \cite{cranston, radoicic}, going back to the proof of the Four Colour Theorem.
As far as we know, the charging rule in \cite{kurz} is the first one that uses Euclidean geometry, as the redistribution of charges in this rule depends on the sizes of angles.
We will use the exact same charging rule.
Compared to the proof in \cite{kurz}, our proof is slightly more technical not only because we need to consider vertices of degree other than $5$ too, but also because we make a careful consideration of the case when the graph is not $2$-regular.

\begin{proof}[Proof of Theorem \ref{thm:small-deg}]
    Since the theorem is easily checked if $n\leq 4$, we assume that $n\geq 5$.
    By the discussion in Section~\ref{planegraphs} of the possible $k$-gons for small $k$, we have $f_0=f_1=f_2=0$, and each $k$-gon with $3\leq k\leq 5$ has a $k$-cycle as boundary.
    The unbounded face cannot be a triangle, otherwise there are no other edges or vertices.
    Thus we have $b\geq 4$, and all triangles are bounded faces.
    We will use the same discharging method as in \cite{kurz}, which corresponds to the vertex charging version of \cite[Proposition~3.1]{cranston}.
    First we assign a charge of $i-6$ to each vertex of degree $i$ and a charge of $2k-6$ to each face of $k$ sides.
    As we show next, the total charge summed over all vertices and faces of the matchstick graph is at most $-12$.
    By counting the edge-face and edge-vertex pairs we find
    \begin{equation}
        b + \sum_{k\geq 3} k f_k = \sum_{i\geq 1} i n_i = 2e. \label{eq:sums2e}
    \end{equation}
    Using Euler's Formula $\sum_i n_i -e + \sum_k f_k = c$, where $c \geq 1$ is the number of connected components, we have
    \begin{equation*}
        -3 \geq -3 \sum_{i\geq 1} n_i +e +2e - 3 \sum_{k\geq 3} f_k
    \end{equation*}
    and using \eqref{eq:sums2e} this gives
    \begin{align}
        -12 &\geq -6 \sum_{i\geq 1} n_i + \sum_{i\geq 1} i n_i + 2\sum_{k\geq 3} k f_k + 2(b-3) - 6 \sum_{k\geq 3} f_k \nonumber\\
        &= \sum_{i\geq 1} (i - 6) n_i + \sum_{k\geq 3} ( 2k - 6 ) f_k + 2b-6. \label{eq:sum}
    \end{align}
    We now redistribute the charge as follows.
    For an angle $\alpha$ between a face and a vertex, transfer from the face to the vertex a charge of
    \begin{equation*}
        \begin{cases}
            0 &\text{ if } \alpha \leq \frac{\pi}{3}, \\
            \frac{3 \alpha}{\pi} - 1 &\text{ if } \frac{\pi}{3} < \alpha \leq \frac{2\pi}{3}, \\
            1 &\text{ if } \alpha > \frac{2\pi}{3}.
        \end{cases}
    \end{equation*}
    \begin{figure}[ht]
        \centering
        \begin{tikzpicture}[xscale=0.75,yscale=2.5]
            \draw[->] (0,0) -- (2*pi+0.8,0) node[below] {$\alpha$};
            \draw[->] (0,0) -- (0,1);
            \draw (pi/3,0) -- (pi/3,-0.05) node[below] {$\frac{\pi}{3}$};
            \draw (2*pi/3,0) -- (2*pi/3,-0.05) node[below] {$\frac{2\pi}{3}$};
            \draw (2*pi,0) -- (2*pi,-0.05) node[below] {$2\pi$};
            \draw (0,1/2) -- (-0.1,1/2) node[left] {$1$};
            \draw (pi/3,0) -- (2*pi/3,1/2);
            \draw (2*pi/3,1/2) -- (2*pi,1/2);
        \end{tikzpicture}
        \caption{Charge transfer from a face to a vertex at an angle of size $\alpha$.}
        \label{fig:charge}
    \end{figure}
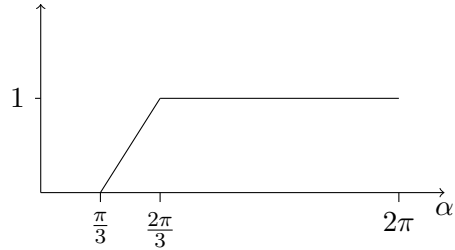
    Since this discharging rule was given without motivation in \cite{kurz}, we try to motivate it now.
    We want the charges of all faces and all vertices of degree at least $5$ to end up non-negative.
    Since triangles start off with a charge of $0$, we would like there to be no transfer of charge from a triangle to its vertices.
    Thus, since all triangles are equilateral, we do not want any transfer of charge in angles of size $\pi/3$.
    Next, consider a hexagonal face.
    In order for this face to end up with a non-negative charge, we have to move a charge of on average at most $1$ to each of its vertices.
    Thus, we cap the charge transfer at $1$, and since the average degree of an angle of a hexagon is $2\pi/3$, we let this maximum charge transfer start at this angle.
    Third, a vertex of degree $5$ needs to get a charge of at least $1$.
    If we let the transferred charge change linearly from $\pi/3$ to $2\pi/3$, then, at least in the case where all angles at the vertex of degree $5$ lie between $\pi/3$ and $2\pi/3$, the vertex receives a total charge of $1$, and thus ends up with a charge of $0$.
    
    We next show that indeed, in all cases, the charges of all faces and all vertices of degree at least $5$ will become non-negative, and we obtain the lower bound~\eqref{eq:-n<-b} for the total charge in terms of $n_i$ ($i=1,2,3,4$) and $b$.
    
    Let us first consider the vertices.
    A vertex of degree $i=1,2,3$ must have an angle $\alpha \geq \frac{2\pi}{3}$ and hence gains at least a charge of $1$ from the redistribution, ending with a charge of at least $i-5$.
    This formula also holds for vertices of degree $4$:
    this is clear if there is an angle $\alpha_j \geq \frac{2\pi}{3}$.
    Otherwise for each angle $\alpha_j$ there is a contribution of at least $\frac{3\alpha_j}{\pi}-1$, so the total contribution is $\frac{3}{\pi}\sum_j \alpha_j - 4 = 2 > 1$.
    For a vertex of degree at least $6$ it is sufficient to note that the initial charge is non-negative, and hence so is the final charge.
    Lastly, for a vertex of degree~$5$ let $s > 0$ be the number of its angles strictly larger than $\frac{\pi}{3}$.
    Since the initial charge is $-1$, if any angle is greater than $\frac{2\pi}{3}$ then the final charge is non-negative.
    Otherwise the sum of these $s$ angles is at least $2\pi - \frac{\pi}{3} (5-s) = \frac{\pi}{3} (s+1)$ and hence the final charge is at least $-1 + \frac{3}{\pi}\cdot\frac{\pi}{3}(s+1)-s=0$.
    From these calculations we can bound the total vertex charge after redistribution from below by
    \begin{equation}
        -4n_1-3n_2 -2n_3 -n_4. \label{eq:vertices}
    \end{equation}
    
    Next consider bounded faces of $k$ sides.
    For $k=3$, a triangular face, the initial charge is 0 and all angles are $\frac{\pi}{3}$ so the final charge stays at 0.
    For $k=4$ the initial charge is $2$.
    Since the face must be a rhombus, each angle is equal to at least one other.
    If exactly two angles are greater than $\frac{2\pi}{3}$ then the final charge is $0$.
    Otherwise all angles must be greater than $\frac{\pi}{3}$ and the final charge is $2-(\frac{3}{\pi} \cdot 2\pi - 4) = 0$.
    For $k=5$ the initial charge is $4$, and as observed before, the face must be a pentagon.
    Let $t$ be the number of angles of the pentagon strictly larger than $\frac{2\pi}{3}$.
    Then the sum of the angles less than $\frac{2\pi}{3}$ is at most $3\pi - \frac{2\pi}{3}t$, with a transferred charge of at most $\frac{3}{\pi} \left ( 3\pi - \frac{2\pi}{3}t \right )-(5-t) = 4-t$.
    Adding the charges of the $t$ angles, totalling to $t$, we get a final transferred charge of at most $4$.
    
    If the number of sides $k$ is at least 6, the initial charge is $2k-6$ and the face can lose up to the maximum possible of $k$, since a $k$-gon has at most $k$ vertices.
    Hence the final charge is at least $k-6\geq 0$.
    In particular, the unbounded face with $b$ edges ends up with a charge of at least $b-6$.
    We can therefore bound below the total charge of faces after redistribution by $b-6$.
    This, together with \eqref{eq:vertices} and~\eqref{eq:sum} gives
    \begin{equation}
        -12 \geq -4n_1-3n_2 -2n_3 -n_4 + b - 6,
        \label{eq:-n<-b}
    \end{equation}
    which rearranges to
    \begin{equation}
        4n_1+3n_2+2n_3+n_4 \geq b+6.
        \label{eq:n>b}
    \end{equation}

    Now we can write Euler's Formula as
    \begin{equation*}
        4n-2e-2e+4\sum_{k \geq 3} f_k \geq 4,
    \end{equation*}
    and using $2e = b+\sum k f_k$ from \eqref{eq:sums2e} we get
    \begin{equation*}
        4n-2e-b+\sum_{k \geq 3}(4-k)f_k \geq 4.
    \end{equation*}
    Rearranging this gives
    \begin{align}
        2e&\leq 4n-4-b+f_3-(f_5+2f_6+\dots) \nonumber \\
        &\leq 4n-8+f_3.
        \label{eq:enf}
    \end{align}
    Suppose for contradiction that the theorem is false. Then
    \begin{equation*}
        4n_1+3n_2+2n_3+n_4 \leq \sqrt{\pi \sqrt{3} \cdot n},
    \end{equation*}
    hence
    \begin{align*}
        2e &= \sum_{i \geq 1} in_i \\
        &\geq n_1+2n_2+3n_3+4n_4+5\sum_{i\geq 5} n_i \\
        &=5n-4n_1-3n_2-2n_3-n_4 \\
        &\geq 5n-\sqrt{\pi\sqrt{3}\cdot n}.
    \end{align*}
    Combining this with \eqref{eq:enf} we get the inequality
    \begin{equation*}
        5n-\sqrt{\pi\sqrt{3}\cdot n} \leq 4n-8+f_3
    \end{equation*}
    and hence
    \begin{equation*}
        f_3 \geq n-\sqrt{\pi\sqrt{3}\cdot n}+8.
    \end{equation*}
    Substituting this into the inequality $b^2 > \pi\sqrt{3}f_3$ from Lemma~\ref{isoperimetric-general} gives
    \begin{equation*}
        b > \sqrt{\pi\sqrt{3}(n-\sqrt{\pi\sqrt{3}\cdot n}+8)}.
    \end{equation*}
    Finally, substituting this into \eqref{eq:n>b} gives
    \begin{align*}
        4n_1+3n_2+2n_3+n_4 &> 6+\sqrt{\pi\sqrt{3} \left (n-\sqrt{\pi\sqrt{3}\cdot n}  +8 \right)} \\
        &> \sqrt{\pi\sqrt{3}\cdot n},
    \end{align*}
which finishes the proof once we verify the last inequality.
Suppose to the contrary that
    \begin{equation*}
        \sqrt{\pi\sqrt{3} \left (n-\sqrt{\pi\sqrt{3}\cdot n} +8\right)} \leq \sqrt{\pi\sqrt{3}\cdot n} - 6.
    \end{equation*}
    Then squaring both sides gives a contradiction
    \begin{align*}
        \pi\sqrt{3}\cdot n-\pi\sqrt{3}\sqrt{\pi\sqrt{3}\cdot n} + 8\pi\sqrt{3} &\leq \pi\sqrt{3}\cdot n - 12\sqrt{\pi\sqrt{3}\cdot n}+36, \\
        (12-\pi\sqrt{3})\sqrt{\pi\sqrt{3}\cdot n} &\leq 36-8\pi\sqrt{3} < 0. \qedhere
    \end{align*}
        
\end{proof}

\begin{proof}[Proof of Corollary \ref{co:n_5}]
    Consider the matchstick graph $G$ and construct its subgraph $G'$ by removing its isolated vertices, i.e. $n_0'=0$.
    Applying the previous theorem on $G'$ we have that
    \begin{equation}
        4n_1'+3n_2'+2n_3'+n_4' > \sqrt{\pi\sqrt{3}\cdot n'}.
        \label{eq:G'}
    \end{equation}
    Clearly removing the isolated vertices does not change the number of vertices of each degree $i\geq1$ hence $n_i'=n_i$.
    We also know that $n'=n-n_0$, so \eqref{eq:G'} gives
    \begin{equation*}
        4n_1+3n_2+2n_3+n_4 > \sqrt{\pi\sqrt{3}(n-n_0)},
    \end{equation*}
    and hence
    \begin{equation*}
        n_1+n_2+n_3+n_4 > \frac{1}{4}\sqrt{\pi\sqrt{3}(n-n_0)}.
    \end{equation*}
    Now adding back the isolated vertices to both sides we obtain
    \begin{align*}
        n_0+n_1+n_2+n_3+n_4 &> \frac{1}{4}\sqrt{\pi\sqrt{3}(n-n_0)} + n_0\\
        &\geq \frac{1}{4}\sqrt{\pi\sqrt{3}\cdot n} - \frac{1}{4}\sqrt{\pi\sqrt{3}\cdot n_0} + n_0\\ & \geq \frac{1}{4}\sqrt{\pi\sqrt{3}\cdot n}. \qedhere 
    \end{align*}
\end{proof}

\section*{Acknowledgement} We thank the referees for their careful proof-reading, for providing the example in Remark~\ref{remark3}, and for improving the exposition of the proof of Theorem~\ref{thm:small-deg}.

\end{document}